\documentclass[12pt,a4paper,oneside]{article}

\usepackage{url}

\usepackage{amsmath}
\usepackage{amsfonts}
\usepackage{amssymb}
\usepackage{graphicx}
\usepackage{multicol}
\usepackage[bottom]{footmisc}
\usepackage{tikz}
\usepackage{mathrsfs}
\usepackage{nicefrac}
\usepackage{dsfont}
\usepackage{soul}
\usepackage{ulem}
\usepackage{amsthm}
\usepackage[colorlinks=true,linkcolor=blue]{hyperref}
\makeatletter
\renewcommand*{\eqref}[1]{\hyperref[{#1}]{\textup{\tagform@{\ref*{#1}}}}}

\usetikzlibrary{calc,backgrounds,fit}
\usetikzlibrary{external}
\usetikzlibrary{patterns}
\usetikzlibrary{decorations.pathreplacing,decorations.markings,arrows}
\usetikzlibrary{shapes,intersections,positioning}
\usetikzlibrary{arrows}

\definecolor{kerstin}{RGB}{0,150,75}

\definecolor{tb}{RGB}{255,102,0}

\newcommand\centre[4][below]{\node (#3) at #2 [circle,minimum size=0.5em,inner sep=0pt,thin,fill,solid] {}; \node [#1=0.1em] at (#3) {#4};}

\newcounter{cst}

\newcommand{\re}{\mathbb{R}}
\newcommand{\R}{\mathbb{R}}

\newcommand{\diver}{\operatorname{div}}

\newcommand{\Tau}{\mathcal{T}}

\newcommand{\dkl}{d_{K|L}}
\newcommand{\edgesint}{\mathcal{E}_{\operatorname{int}}}
\newcommand{\edgeskint}{\mathcal{E}_{K,\operatorname{int}}}
\newcommand{\edgeskiint}{\mathcal{E}_{K_i,\operatorname{int}}}

\newcommand{\tn}{t_n}
\newcommand{\tnp}{t_{n+1}}
\newcommand{\dlt}{\Delta t}

\newcommand{\unpk}{u^{n+1}_K}

\newcommand{\ve}{\mathbf{v}}
\newcommand{\ms}{m_{\sigma}}

\newcommand{\vksnp}{v^{n+1}_{K,\sigma}}

\newtheorem{definition}{Definition}[section]

\newtheorem{proposition}[definition]{Proposition}

\theoremstyle{remark}
\newtheorem{remark}[definition]{Remark}

\title{Finite Volume Approximations for Non-Linear Parabolic Problems with Stochastic Forcing}
\author{Caroline Bauzet\footnotemark[1] \and Flore Nabet\footnotemark[2] \and Kerstin Schmitz\footnotemark[3] \and Aleksandra Zimmermann\footnotemark[4]}

\begin{document}

\maketitle

\begin{abstract}
We propose a two-point flux approximation finite-volume scheme for a stochastic non-linear parabolic equation with a multiplicative noise.
The time discretization is implicit except for the stochastic noise term in order to be compatible with stochastic integration in the sense of It\^{o}.
We show existence and uniqueness of solutions to the scheme and the appropriate measurability for stochastic integration follows from the uniqueness of approximate solutions.\\

\textbf{Keywords:} stochastic non-linear parabolic equation, multiplicative Lipschitz noise, finite-volume method, upwind scheme, diffusion-convection equation, variational approach\\

\textbf{Mathematics Subject Classification (2020):} 65M08, 60H15, 35K55
\end{abstract}

\footnotetext[1]{Aix Marseille Univ, CNRS, Centrale Marseille, LMA UMR 7031, Marseille, France, caroline.bauzet@univ-amu.fr}
\footnotetext[2]{CMAP, CNRS, \'Ecole polytechnique, Institut Polytechnique de Paris, 91120 Palaiseau, France, flore.nabet@polytechnique.edu}
\footnotetext[3]{Universit\"at Duisburg-Essen, Fakult\"at f\"ur Mathematik, Essen, Germany, kerstin.schmitz@uni-due.de}
\footnotetext[4]{TU Clausthal, Institut f\"ur Mathematik, Clausthal-Zellerfeld, Germany, aleksandra.zimmermann@tu-clausthal.de}

\section{Introduction}
Let $\Lambda$ be a bounded, open, connected and  polygonal set of $\R^2$.
Moreover let $(\Omega,\mathcal{A},\mathds{P})$ be a probability space endowed with a right-continuous, complete filtration $(\mathcal{F}_t)_{t\geq 0}$ and let $(W(t))_{t\geq 0}$ be a standard, one-dimensional Brownian motion with respect to $(\mathcal{F}_t)_{t\geq 0}$ on $(\Omega,\mathcal{A},\mathds{P})$.\\
For $T>0$, we consider the following  non-linear parabolic problem forced by  a multiplicative stochastic noise:
\begin{align}\label{equation}
\begin{aligned}
du-\Delta u\,dt + \text{div} \big(\mathbf{v}f(u)\big)\, dt &=g(u)\,dW(t)+\beta(u)\,dt, &&\text{in }\Omega\times(0,T)\times\Lambda;\\
u(0,.)&=u_0, &&\text{in } \Omega\times\Lambda;\\
\nabla u\cdot \mathbf{n}&=0, &&\text{on }\Omega\times(0,T)\times\partial\Lambda;
\end{aligned}
\end{align}
where $\diver$ is the divergence operator with respect to the space variable and $\mathbf{n}$ denotes the unit normal vector to $\partial\Lambda$ outward to $\Lambda$.
After setting $L_f, L_\beta$ and $L_g$ in $\R_+^*$, we assume the following hypotheses on the data:
\begin{itemize}
\item[$H_1$:] $u_0\in L^2(\Omega;H^1(\Lambda))$ is $\mathcal{F}_0$-measurable.
\item[$H_2$:] $f:\R\rightarrow \R$ is nondecreasing, $L_f$-Lipschitz continuous with $f(0)=0$.
\item[$H_3$:] $g:\re\rightarrow\re$ is a $L_g$-Lipschitz continuous function.
 \item[$H_4$:] $\beta : \R\rightarrow \R$  is  $L_\beta$-Lipschitz continuous with $\beta(0)=0$.
\item[$H_5$:] $\mathbf{v}\in \mathscr{C}^1([0,T]\times \Lambda; \R^2)$ such that $\text{div}(\mathbf{v})=0$ in $[0,T]\times \Lambda$ and $\mathbf{v}\cdot\mathbf{n}=0$ on $[0,T]\times \partial\Lambda$.
\end{itemize}

\subsection{Concept of Solution and Main Result}

We will be interested in the concept of solution as defined below, which we will call a variational solution: 
\begin{definition}\label{solution} A predictable stochastic  process $u$ is a variational solution to Problem \eqref{equation} if it belongs to 
\begin{align*}
 L^2(\Omega;\mathscr{C}([0,T];L^2(\Lambda)))\cap L^2(\Omega;L^2(0,T;H^1(\Lambda)))
\end{align*}
and satisfies, for all $t\in[0,T]$, in $L^2(\Lambda)$, and $\mathds{P}$-a.s. in $\Omega$
\begin{align*}
&u(t)-u_0-\int_0^t \Delta u(s)\,ds+\int_0^t \diver\big(\mathbf{v}(s,.)f(u(s))\big)\, ds\\
&=\int_0^t g(u(s))\,dW(s)+\int_0^t \beta(u(s))\,ds.
\end{align*}
\end{definition}

Existence, uniqueness and regularity of this variational solution is well-known in the literature, see, e.g., \cite{KryRoz81}.

\subsection{Outline}
In this contribution, we propose a finite-volume approximation scheme for the solution of \eqref{equation} in the sense of Definition \ref{solution}. We show existence and uniqueness of solutions to the scheme. In Section \ref{sectiontwo}, we introduce the notation for our finite-volume framework. In Section \ref{FVscheme}, we introduce our finite-volume scheme. The main result is contained in Section \ref{mainres}.

\section{Admissible Finite-Volume Meshes and Notations}\label{sectiontwo}

In order to perform a finite-volume approximation of the variational solution of Problem \eqref{equation} on $[0,T]\times \Lambda$ we need first of all to set a choice for the temporal and spatial discretization. For the time-discretization, let $N\in \mathbb{N}^*$ be given. We define the  fixed time step $\Delta t=\frac{T}{N}$ and  divide the interval $[0,T]$ in $0=t_0<t_1<...<t_N=T$ equidistantly with $\tn=n \Delta t$ for all $n\in \{0, ..., N\}$.
For the space discretization, we refer to \cite{gal} and consider finite-volume admissible meshes in the sense of the following definition.
\begin{definition}[Admissible finite-volume mesh]\label{defmesh}
An admissible finite-volume mesh $\mathcal{T}$ of $\Lambda$ (see Fig.~\ref{fig:notation_mesh}) is given by a family of open polygonal and convex subsets $K$, called \textit{control volumes} of $\Tau$, satisfying the following properties:
\begin{itemize}
\item $\overline{\Lambda}=\bigcup_{K\in\Tau}\overline{K}$.
\item If $K,L\in\Tau$ with  $K\neq L$ then $\operatorname{int}K\cap\operatorname{int}L=\emptyset$.
\item If $K,L\in\Tau$, with $K\neq L$ then either the $1$-dimensional  Lebesgue measure of $\overline{K}\cap \overline{L}$ is $0$ or $\overline{K}\cap \overline{L}$ is the edge, denoted by $\sigma=K|L$, separating the control volumes $K$ and $L$.
\item To each control volume $K\in\Tau$, we associate a point $x_K\in \overline{K}$ (called the center of $K$) such that: If $K,L\in\Tau$ are two neighbouring control volumes the straight line between the centers $x_K$ and $x_L$ is orthogonal to the edge $\sigma=K|L$.
\end{itemize}
\end{definition}

\begin{figure}[htbp!]
\centering
\begin{tikzpicture}[scale=1.8]

  \clip (-1.2,-0.6) rectangle (1.8,1.3);

  \node[rectangle,fill] (A) at (-1,0.6) {};
  \node[rectangle,fill] (B) at (0,1.2) {};
  \node[rectangle,fill] (C) at (0,-0.2) {};
  \node[rectangle,fill] (D) at (1.5,0.3) {};

  \centre[above right]{(-0.6,0.5)}{xK}{$x_K$};
  \centre[above left]{(0.9,0.5)}{xL}{$x_L$};
  
  \draw[thick] (B)--(C) node [pos=0.7,right] {$\sigma=$\small{$K|L$}};

  \draw[thin,opacity=0.5] (A) -- (B) -- (C) -- (A) ;
  \draw[thin,opacity=0.5] (D) -- (B) -- (C) -- (D);

  \draw[dashed] (xK) -- (xL);
  
  \coordinate (KK) at ($(xK)!0.65!-90:(xL)$);
  \coordinate (LL) at ($(xL)!0.65!90:(xK)$);

  \draw[dotted,thin] (xK) -- (KK);
  \draw[dotted,thin] (xL) -- (LL);

  \draw[|<->|] (KK) -- (LL) node [midway,fill=white,sloped] {$\dkl$};
 
  \coordinate (KAB) at ($(A)!(xK)!(B)$);
  \coordinate (KAC) at ($(A)!(xK)!(C)$);

  \coordinate (LDB) at ($(D)!(xL)!(B)$);
  \coordinate (LDC) at ($(D)!(xL)!(C)$);

  \draw[dashed] (xK) -- ($(xK)!3!(KAB)$);
  \draw[dashed] (xK) -- ($(xK)!3!(KAC)$);

  \draw[dashed] (xL) -- ($(xL)!3!(LDB)$);
  \draw[dashed] (xL) -- ($(xL)!3!(LDC)$);

  \begin{scope}[on background layer]   
    \draw (0,0.5) rectangle ++ (0.1,-0.1);
  \end{scope}
  
  \coordinate (nkl) at ($(B)!0.3!(C)$);
  \draw[->,>=latex] (nkl) -- ($(nkl)!0.3cm!90:(C)$) node[below] {$\mathbf{n}_{K,\sigma}$};

\end{tikzpicture}
\caption{Notations of the mesh $\mathcal T$ associated with $\Lambda$\label{fig:notation_mesh}}
\end{figure}
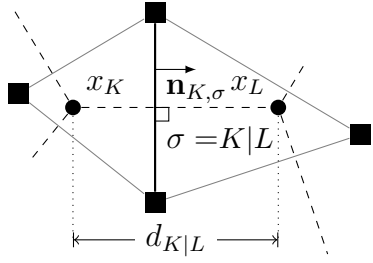
Once an admissible finite-volume mesh $\Tau$ of $\Lambda$ is fixed, we will use the following notations.
\subsection{Notation}
\begin{itemize}
\item $h=\operatorname{size}(\Tau)=\sup\{\operatorname{diam}(K): K\in\Tau\}$, the mesh size.
\item $d_h\in\mathbb{N}$ the number of control volumes $K\in\Tau$ with $h=\operatorname{size}(\Tau)$.
\item $\edgesint:=\{\sigma:\sigma\nsubseteq \partial\Lambda\}$ is the set of interior edges of the mesh $\Tau$.
\item For $K\in\Tau$, $\mathcal{E}_K$ is the set of edges of $K$, $\mathcal{E}_{K,\operatorname{int}}=\mathcal{E}_K\cap\edgesint$ and $m_K$ is the Lebesgue measure of $K$.
\item For $K\in\Tau$ and $\sigma \in \mathcal{E}_K$, $\mathbf{n}_{K,\sigma}$ is the unit normal vector to $\sigma$ outward to $K$.
\item Let $K,L\in\Tau$ be two neighbouring control volumes. For $\sigma=K|L\in\edgesint$, let $m_\sigma$ be the length of $\sigma$ and $d_{K|L}$ the distance between $x_K$ and $x_L$.
\item For any vector  $u_h= (u_K)_{K\in\Tau} \in \re^{d_h}$, we define the $L^2$-norm on $\Lambda$ by
$$
\|u_h\|_{L^2(\Lambda)} = \left(\sum_{K\in\Tau} m_K |u_K|^2 \right)^\frac12.
$$
\end{itemize}
In the sequel, we note $|x|$ the euclidian norm of $x \in \R^d$ with $d\geq 1$.

\section{The Finite-Volume Scheme}\label{FVscheme}

Firstly, we define the vector  $u_h^0= (u^0_K)_{K\in\Tau} \in \re^{d_h}$ by the discretization of the initial condition $u_0$ of Problem \eqref{equation} over each control volume:
\begin{align}
\label{eq:def_u0}
u_K^0:=\frac{1}{m_K}\int_K u_0(x)\,dx, \quad \forall K\in \Tau.
\end{align}
The finite-volume scheme we propose reads, for this given initial $\mathcal{F}_0$-measurable random vector $u_h^0\in\re^{d_h}$: \\
For any $n \in \{0,\dots,N-1\}$, knowing $u_h^n= (u^{n}_K)_{K\in\Tau} \in \re^{d_h}$ we search for $u_h^{n+1}=(u^{n+1}_K)_{K\in\Tau}\in\re^{d_h}$ such that, for almost every $\omega\in\Omega$, the vector $u_h^{n+1}$ is solution to the following random equations 
\begin{align}\label{equationapprox}
\begin{split}
&\frac{m_K}{\Delta t}(u_K^{n+1}-u_K^n)+\sum_{\sigma=K|L\in\edgeskint} m_\sigma \vksnp f(u^{n+1}_{\sigma}) \\
&+\sum_{\sigma=K|L\in\edgeskint}\frac{m_\sigma}{\dkl}(u_K^{n+1}-u_L^{n+1})\\
&=\frac{m_K}{\Delta t}g(u_K^n)(W^{n+1}-W^n)+m_K \beta(\unpk),\quad \forall K\in \Tau,
\end{split}
\end{align}
where, by denoting $\gamma$ the $(d-1)$-dimensional Lebesgue measure,
$$\vksnp=\frac{1}{\dlt \ms}\int_{\tn}^{\tnp}\int_{\sigma}\ve(t,x)\cdot\mathbf{ n}_{K,\sigma}\, d\gamma(x)dt,$$ 
and $u^{n+1}_{\sigma}$ denotes the upstream value at time $t_{n+1}$ with respect to $\sigma$ defined as follows: If $\sigma=K|L\in\edgeskint$ is the interface between the control volumes $K$ and $L$, $u^{n+1}_{\sigma}$ is equal to $u^{n+1}_K$ if $\vksnp\geqslant 0$ and to $u^{n+1}_L$ if $\vksnp< 0$.
Note also that $W^{n+1}-W^n=W(t_{n+1})-W(t_n)$ for $n\in\{0,\dots,N-1\}$. 

\begin{remark} Since $\text{div}(\mathbf{v})=0$ in $[0,T]\times \Lambda$, for any $n \in \{0,\cdots,N-1\}$ and $K\in\Tau$ one has $ \sum_{\sigma=K|L\in\edgeskint} m_\sigma \vksnp = 0$. Thus, using that $\vksnp=(\vksnp)^+-(\vksnp)^-$ (where, for $r\in\re$, $r^+:=\max\{r,0\}$ and  $r^-:=-\min\{0,r\}$) an equivalent formulation of the scheme \eqref{equationapprox} is given by
\begin{align}\label{equationapproxter}
\begin{split}
&\frac{m_K}{\Delta t}(u_K^{n+1}-u_K^n)+\sum_{\sigma=K|L\in\edgeskint} m_\sigma (\vksnp)^-\Big(f(u^{n+1}_{K})-f(u^{n+1}_L)\Big)\\
&+\sum_{\sigma=K|L\in\edgeskint}\frac{m_\sigma}{\dkl}(u_K^{n+1}-u_L^{n+1})\\
&=\frac{m_K}{\Delta t}g(u_K^n)\left(W^{n+1}-W^n\right)+m_K \beta(\unpk), \quad \forall K \in \Tau.
\end{split}
\end{align}
\end{remark}

\section{Main Result}\label{mainres}
\begin{proposition}[Existence of a discrete solution]
\label{210609_prop1}
Assume that hypotheses $H_1$ to $H_5$ hold. Let $\Tau$ be an admissible finite-volume mesh of $\Lambda$  in the sense of Definition \ref{defmesh} with a mesh size $h$ and $N\in \mathbb{N}^*$. Then, there exists a unique solution $(u_h^n)_{1\le n \le N} \in (\re^{d_h})^N$ to Problem~\eqref{equationapprox} associated with the initial vector $u^0_h$ defined by~\eqref{eq:def_u0}. Additionally,  for any $n\in \{0,\ldots,N\}$, $u_h^n$ is a $\mathcal{F}_{t_n}$-measurable random vector.
\end{proposition}

\begin{proof}
We fix $n\in \{0,\ldots, N-1\}$ and choose an arbitrary vector $u_h^n=(u^n_K)_{K\in \Tau}\in \mathbb{R}^{d_h}$. Firstly, we will show that there exists at least one random vector $u_h^{n+1}=(u^{n+1}_K)_{K\in \Tau}\in \mathbb{R}^{d_h}$ such that \eqref{equationapproxter} holds true  $\mathds{P}$-a.s in $\Omega$. To this end, we define the mapping $\mathbf{P}^n:\mathbb{R}^{d_h}\rightarrow \mathbb{R}^{d_h}$, $\mathbf{P}^n=(P_1^n,\ldots,P_{d_h}^n)$ such that for any $i\in \{1,\ldots, d_h\}$
\begin{align*}
P_i^n(w_{K_1},\ldots,w_{K_{d_h}})
=\ & \frac{m_{K_i}}{\Delta t} w_{K_i} 
- m_{K_i}\beta(w_{K_i}) \\
&+\sum_{\sigma=K_i|K_j \in \edgeskiint} m_{\sigma}(v_{K_i,\sigma}^{n+1})^{-}(f(w_{K_i})-f(w_{K_j}))\\
&+\sum_{\sigma=K_i|K_j \in \edgeskiint} \frac{m_{\sigma}}{d_{K_i|K_j}}(w_{K_i}-w_{K_j})- \frac{m_{K_i}}{\Delta t} \xi_i^n 
\end{align*}
where $\xi_i^n:=u_{K_i}^n+g(u_{K_i}^n)(W^{n+1}-W^n)$. 

Obviously, $\mathbf{P}^n$ is a continuous mapping. Next, we show that there exists $\varrho>0$ such that for all $w_h=(w_{K_i})_{1\le i\le d_h}\in\mathbb{R}^{d_h}$ such that $|w_h|=\varrho$,
\begin{equation*}
(\mathbf{P}^n(w_h),w_h)_{\mathbb{R}^{d_h}}:=\sum_{i=1}^{d_h}P_i^n(w_h)w_{K_i}\geq 0.
\end{equation*}
In this case, from \cite[Lemma 4.3]{Lions} it follows that there exists at least one $\overline{w}_h\in\mathbb{R}^{d_h}$ such that $| \overline{w}_h|\leq\varrho$ and $\mathbf{P}^n(\overline{w}_h)=0$. We have
\begin{align*}
\sum_{i=1}^{d_h} P_i^n(w_h)w_{K_i}
&=\sum_{i=1}^{d_h}\frac{m_{K_i}}{\Delta t}w_{K_i}^2
-\sum_{i=1}^{d_h}m_{K_i}\beta(w_{K_i})w_{K_i}
-\sum_{i=1}^{d_h}\frac{m_{K_i}}{\Delta t}\xi_i^n w_{K_i} \\
&\quad+\sum_{i=1}^{d_h}\sum_{\sigma=K_i|K_j \in \edgeskiint} m_{\sigma}(v_{K_i,\sigma}^{n+1})^{-}
(f(w_{K_i})-f(w_{K_j}))w_{K_i}
\\
&\quad+\sum_{i=1}^{d_h}\sum_{\sigma=K_i|K_j \in \edgeskiint}\frac{m_{\sigma}}{d_{K_i|K_j}}(w_{K_i}-w_{K_j})w_{K_i}\\
&=:I_1+I_2+I_3+I_4+I_5.
\end{align*}
Since $\beta$ is Lipschitz continuous, the term $I_2$ satisfies
\begin{equation}
    \label{term_I2}
    I_2 \geq - L_\beta \|w_h\|^2_{L^2(\Lambda)}.
\end{equation}
Moreover, by discrete partial integration,
\begin{align}\label{term_I5}
I_5=\sum_{\sigma=K_i|K_j \in \edgeskiint}\frac{m_{\sigma}}{d_{K_i|K_j}}|w_{K_i}-w_{K_j}|^2\geq 0.
\end{align}
Now, we focus on the term $I_4$. 
Since $f$ is Lipschitz continuous and nondecreasing, thanks to \cite[Lemma 18.5]{gal}, for any $r\in\mathbb{R}$, using the notation $\Phi(r)=\int_0^r f'(s)s\,ds$, for any $a,b\in\mathbb{R}$, one has
\begin{align*}
b(f(b)-f(a)) &= \int_a^b (sf(s))' ds - (b-a)f(a)\\
&= \int_a^b \Phi'(s) ds + \int_a^b (f(s)-f(a)) ds \\
&\geq (\Phi(b)-\Phi(a))+\frac1{2L_f} (f(b)-f(a))^2.
\end{align*}
Thus, since $\diver \ve=0$ in $[0,T]\times\Lambda$ and $v_{K_i,\sigma}^{n+1}=-v_{K_j,\sigma}^{n+1}$, we obtain
\begin{equation}\label{term_I4}
\begin{aligned}
I_4 &\geq \sum_{i=1}^{d_h} \sum_{\sigma=K_i|K_j \in \edgeskiint} m_{\sigma} (v_{K_i,\sigma}^{n+1})^{-}(\Phi(w_{K_i})-\Phi(w_{K_j})) \\
&=\sum_{i=1}^{d_h} \Phi(w_{K_i}) \sum_{\sigma=K_i|K_j \in \edgeskiint} m_{\sigma}(v_{K_i,\sigma}^{n+1})
=0.
\end{aligned}
\end{equation}
For the term $I_3$, since $-ab\ge-\frac{1}{2}(a^2+b^2)$ one has
\begin{align}\label{term_I3}
I_3
\geq -\frac1{2\Delta t} \left(\|w_h\|^2_{L^2(\Lambda)}+\|\xi_h^n\|^2_{L^2(\Lambda)}\right).
\end{align}
From \eqref{term_I2}, \eqref{term_I5}, \eqref{term_I4} and \eqref{term_I3} for some $\alpha>0$, choosing $\Delta t \leq \frac1{2 (\alpha+L_\beta)}$ we now get
\begin{align}\label{230227_01}
\begin{aligned}
\sum_{i=1}^{d_h} P_i^n(w_h)w_{K_i} 
&\geq \frac1{2\Delta t} \|w_h\|^2_{L^2(\Lambda)} 
-L_\beta\|w_h\|^2_{L^2(\Lambda)}
-\frac1{2\Delta t} \|\xi_h^n\|^2_{L^2(\Lambda)} \\
&\geq \alpha(\min_{K\in\Tau} m_K) |w_h|^2 -\frac1{2\Delta t} \|\xi_h^n\|^2_{L^2(\Lambda)}.
\end{aligned}
\end{align}
Then, setting
 \[\varrho:=\sqrt{\frac{1}{2\alpha(\min_{K\in\Tau} m_K)\Delta t}}\|\xi_h^n\|_{L^2(\Lambda)}>0\]
we get $(\mathbf{P}^n(w_h),w_h)_{\mathbb{R}^{d_h}}\geq 0$ from \eqref{230227_01} for all $w_h\in\mathbb{R}^{d_h}$ such that $|w_h|=\varrho$. Hence, there exists at least one element $\overline w_h$ such that $\mathbf{P}^n(\overline{w}_h)=0$. Thus, $u_h^{n+1}:=\overline w_h\in\mathbb{R}^{d_h}$ is solution to the numerical scheme~\eqref{equationapproxter}.

Next, we will prove the uniqueness of the solution. Therefore, we assume that there exist $w_h=(w_{K_i})_{1\le i \le d_h}\in\R^{d_h}$ and $z_h=(z_{K_i})_{1\le i \le d_h}\in\R^{d_h}$ satisfying $\mathbf{P}^n({w}_h)=\mathbf{P}^n({z}_h)=0$. Taking $P_i^n(w_h)-P_i^n(z_h)$, and using the initial formulation of the scheme \eqref{equationapprox}, for any $i=1,\dots,d_h$ we obtain
\begin{align*}
&\frac{m_{K_i}}{\Delta t}(w_{K_i}-z_{K_i})-m_{K_i}(\beta(w_{K_i})-\beta(z_{K_i}))\\
&+\sum_{\sigma=K_i|K_j \in \edgeskiint}m_{\sigma}v_{K_i,\sigma}^{n+1}(f(w_{\sigma})-f(z_{\sigma}))\\
&+\sum_{\sigma=K_i|K_j \in \edgeskiint}\frac{m_{\sigma}}{d_{K_i|K_j}}\left((w_{K_i}-w_{K_j})-(z_{K_i}-z_{K_j})\right)\\
&=0,
\end{align*}
where $w_\sigma$ and $z_\sigma$ are the upstream value with respect to $\sigma$.

Now, we adjust the method developed in the proof of \cite[Proposition~26.1]{gal}: Using the monotonicity of $f$, the fact that $v_{K_i,\sigma}^{n+1}=(v_{K_i,\sigma}^{n+1})^+ - (v_{K_i,\sigma}^{n+1})^- $ and taking the absolute value, one has
\begin{align}\label{230227_04}
\begin{aligned}
&\frac{m_{K_i}}{\Delta t}|w_{K_i}-z_{K_i}|
+\sum_{\sigma=K_i|K_j \in \edgeskiint} \frac{m_{\sigma}}{d_{K_i|K_j}}|w_{K_i}-z_{K_i}| \\
&+\sum_{\sigma=K_i|K_j \in \edgeskiint} m_{\sigma} (v_{K_i,\sigma}^{n+1})^+ |f(w_{K_i})-f(z_{K_i})|  \\
&\leq  m_{K_i}|\beta(w_{K_i})-\beta(z_{K_i})|
+ \sum_{\sigma=K_i|K_j \in \edgeskiint} \frac{m_{\sigma}}{d_{K_i|K_j}} |w_{K_j}-z_{K_j}| \\
&\quad+ \sum_{\sigma=K_i|K_j \in \edgeskiint} m_{\sigma} (v_{K_i,\sigma}^{n+1})^- |f(w_{K_j})-f(z_{K_j})|.
\end{aligned}
\end{align}
For $\eta>0$, $x\in\mathbb{R}^2$ we define $\varphi(x)=\exp(-\eta|x|)$ and for $K_i\in\Tau$, $i=1,\ldots,d_h$ let 
\[\varphi_{K_i}:=\frac{1}{m_{K_i}}\int_{K_i}\varphi(x)\,dx.\]
Multiplying \eqref{230227_04} by $\varphi_{K_i}$, taking the sum over $i=1,\ldots,d_h$ and rearranging the sums on the right-hand side by fixing $j$ and varying over $i$ we obtain
\begin{align}\label{230227_05}
\begin{aligned}
&\sum_{i=1}^{d_h}\frac{m_{K_i}}{\Delta t} \varphi_{K_i}  |w_{K_i}-z_{K_i}|
 +\sum_{i=1}^{d_h} \varphi_{K_i} \sum_{\sigma=K_i|K_j \in \edgeskiint} \frac{m_{\sigma}}{d_{K_i|K_j}} |w_{K_i}-z_{K_i}| \\
&+ \sum_{i=1}^{d_h} \varphi_{K_i} \sum_{\sigma=K_i|K_j \in \edgeskiint}  m_{\sigma}(v_{K_i,\sigma}^{n+1})^+ |f(w_{K_i})-f(z_{K_i})|\\
&\leq I_1 + I_2 + I_3
\end{aligned}
\end{align}
where
\begin{align}\label{230227_06}
\begin{aligned}
I_2
\leq& \sum_{i=1}^{d_h}|w_{K_i}-z_{K_i}| \sum_{\sigma=K_i|K_j \in \edgeskiint} \frac{m_{\sigma}}{d_{K_i|K_j}}| \varphi_{K_i}-\varphi_{K_j}| \\
&+ \sum_{i=1}^{d_h} |w_{K_i}-z_{K_i}| \sum_{\sigma=K_i|K_j \in \edgeskiint} \frac{m_{\sigma}}{d_{K_i|K_j}} \varphi_{K_i}
\end{aligned}
\end{align}
and similarly, since $(v_{K_j,\sigma}^{n+1})^{-}=(-v_{K_i,\sigma}^{n+1})^{-}=(v_{K_i,\sigma}^{n+1})^{+}$ for $\sigma=K_i|K_j$, using the Lipschitz continuity of $f$
\begin{align}\label{230227_07}
\begin{aligned}
I_3&\leq \sum_{i=1}^{d_h}L_f|w_{K_i}-z_{K_i}|\sum_{\sigma=K_i|K_j \in \edgeskiint}m_{\sigma}(v_{K_j,\sigma}^{n+1})^{-}|\varphi_{K_i}-\varphi_{K_j}|\\
&\quad+\sum_{i=1}^{d_h}|f(w_{K_i})-f(z_{K_i})|\sum_{\sigma=K_i|K_j \in \edgeskiint}m_{\sigma}(v_{K_i,\sigma}^{n+1})^{+}\varphi_{K_i}.
\end{aligned}
\end{align}
Now, plugging \eqref{230227_06} and \eqref{230227_07} into \eqref{230227_05} and using the Lipschitz continuity of $\beta$ we obtain for all $i=1,\ldots,d_h$
\begin{align}\label{230227_08}
\sum_{i=1}^{d_h}a_i|w_{K_i}-z_{K_i}|\leq \sum_{i=1}^{d_h}b_i|w_{K_i}-z_{K_i}|
\end{align}
with
\begin{align*}
a_i&:= m_{K_i} \left(\frac1{\Delta t} - L_\beta\right) \varphi_{K_i}\\
b_i&:=\sum_{\sigma=K_i|K_j \in \edgeskiint}\left(\frac{m_{\sigma}}{d_{K_i|K_j}}+m_{\sigma}(v_{K_j,\sigma}^{n+1})^{-}\right)|\varphi_{K_i}-\varphi_{K_j}|.
\end{align*}
Now, taking $\Delta t \leq \frac1{2L_\beta}$ using the same arguments as in the proof of \cite[Proposition 26.1]{gal}, we may choose $\eta>0$ small enough such that $a_i>b_i$ for all $i=1,\ldots d_h$. 
Thus $w_{K_i}=z_{K_i}$ then follows from \eqref{230227_08} for all $i=1,\ldots,d_h$. Since the initial vector $u_h^0$ is given, the existence of a unique solution $(u_h^n)_{1\leq n\leq N}\in\mathbb{R}^{d_h}$ follows by iteration. 

It is left to prove that $u_h^n$ is a $\mathcal{F}_{t_{n}}$-measurable random vector for all $n=1,\ldots,N$. 
We have already shown that for any given $\xi_h\in\mathbb{R}^{d_h}$ there exists a unique $w_h=(w_{K_i})_{1 \le i \le d_h}\in\mathbb{R}^{d_h}$ such that
$\mathbf{Q}(w_h)=\xi_h$ where $\mathbf{Q}:\mathbb{R}^{d_h}\rightarrow\mathbb{R}^{d_h}$, $\mathbf{Q}(w_h)=(Q_1,\ldots,Q_{d_h})(w_h)$ is defined by $Q_i(w_h)=P_i(w_h)-\frac{m_{K_i}}{\Delta t}\xi_i^n$
for all $i=1,\ldots,d_h$. Thus, $u_h^{n+1}=\mathbf{Q}^{-1}(\xi_h^n)$ $\mathds{P}$-a.s. in $\Omega$, where $\xi_h^n=(\xi_1^n,\ldots,\xi_{d_h}^n)$. Since $\mathbf{Q}^{-1}$ is continuous, if $\xi_h^n$ is $\mathcal{F}_{t_{n+1}}$-measurable the same holds true for $u_h^{n+1}$. 
Indeed, let $(\zeta^k)_k\subset \mathbb{R}^{d_h}$ be a sequence such that $\zeta^k\rightarrow \zeta$ for some $\zeta\in\mathbb{R}^{d_h}$ for $k\rightarrow\infty$. Then, for $w^k:=\mathbf{Q}^{-1}(\zeta^k)$ from \eqref{230227_01} and from the theorem of Bolzano-Weierstrass it follows that there exists $w\in\mathbb{R}^{d_h}$ such that, passing to an unlabelled subsequence if necessary, $w^k\rightarrow w$ for $k\rightarrow\infty$. This strong convergence is enough to pass to the limit in $\mathbf{Q}(w^k)$, and therefore $\mathbf{Q}(w)=\zeta$. Thanks to uniqueness, we get convergence of the whole sequence $(w^k)_k$, hence $\lim_{k\rightarrow\infty}\mathbf{Q}^{-1}(\zeta^k)=\mathbf{Q}^{-1}(\zeta)$ and $\mathbf{Q}^{-1}$ is continuous.
\end{proof}

\textbf{Acknowledgments} The authors would like to thank the German Research Foundation, the Institut de M\'ecanique et d'Ingenierie of Marseille, the Programmes for Project-Related Personal Exchange Procope and Procope Plus and the Procope Mobility Program DEU-22-0004 LG1 for financial support.

\end{document}